\documentclass{amsart}
\usepackage{amscd}
\usepackage{amsfonts}
\usepackage{amsthm}
\usepackage{geometry}
\usepackage{amsmath}
\usepackage{txfonts}
\usepackage{graphicx}
\usepackage{wrapfig}
\usepackage{ascmac}
\usepackage{bm}
\usepackage{indentfirst}

\usepackage{enumitem}
\usepackage{changes}

\newtheorem{thm}{Theorem}[section]
\newtheorem{rem}[thm]{Remark}

\newtheorem{lem}[thm]{Lemma}

\def\NZQ{\Bbb}

\def\RR{{\NZQ R}}
\def\CC{{\NZQ C}}

\def\PP{{\NZQ P}}

\def\SS{{\NZQ S}}
\def\TT{{\NZQ T}}

\def\ml{\mathcal{C}}
\def\ml1{\mathcal{C}^1}
\def\mlb1{\mathcal{C}_{b}^{1}}

\def\frk{\frak}

\def\Phi{{\frk n}}

\def\sign{{\rm sign}}

\def\dim{{\rm dim}}
\def\rank{{\rm rank}}

\def\A{{\mathcal A}}
\def\B{{\mathcal B}}

\newcommand{\R}[0]{\mathbb{R}}
\newcommand{\C}[0]{\mathbb{C}}

\newcommand{\Zt}[0]{\mathcal Z}

\title[Discriminantal arrangemet and hypersurfaces in Grassmannian]{Discriminantal Arrangement, $3\times3$ Minors of Pl\"{u}cker Matrix and hypersurfaces in Grassmannian $Gr(3,n)$.}
\author{S. Sawada}
\author{ S. Settepanella}
 \author{S. Yamagata}
\address{%
Department of Mathematics,
Hokkaido University, Japan.}
\email{b.lemon329@gmail.com}
\email{s.settepanella@math.sci.hokudai.ac.jp}
\email{so.yamagata.math@gmail.com}

\thanks{The second named author was supported by JSPS Kakenhi Grant Number 26610001.}


\begin{document}

\maketitle

%%%%%%%%%%%%%%%%%%%%%%%%%%%%%%%%%%%%%%%%%%%%%%%%%%%%%%%%%%%%%%%%%%%%%%

\begin{abstract}
We show that points in specific degree 2 hypersurfaces in the Grassmannian $Gr(3, n)$ correspond to generic arrangements of $n$ hyperplanes in $\C^3$ with associated discriminantal arrangement having intersections of multiplicity three in codimension two.

\end{abstract}

\section{Introduction}

In 1989 Manin and Schechtman (cf.\cite{man}) 
considered a family of arrangements of hyperplanes 
generalizing classical braid arrangements which they called 
the {\it discriminantal arrangements} (cf. \cite{man} p.209). 
Such  an arrangement $\B(n,k), n,k \in {\bf N}$ 
for $k \ge 2$ depends on a choice
$H^0_1,...,H^0_n$ of collection of hyperplanes in 
general position in $\CC^k$. It consists of 
parallel translates of $H_1^{t_1},...,H_n^{t_n}, (t_1,...,t_n) \in \CC^n$ 
which fail to form a generic arrangement in $\CC^k$.
$\B(n,k)$ can be viewed as a  
generalization of the pure braid group arrangement (cf. \cite{OT}) with which  
$\B(n,1)$ coincides. 
These arrangements have several beautiful 
relations with diverse problems including combinatorics (cf. \cite{man}, \cite{athana}, \cite{BB} and also 
\cite{Crapo}, which is an earlier appearance of discriminantal
arrangmements), Zamolodchikov equation with its relation to 
 higher category theory (cf. Kapranov-Voevodsky \cite{Kap}),
and vanishing of cohomology of bundles on toric varieties
(cf. \cite{Perling}). 
Paper \cite{man} concerns with arrangements
$\B(n,k)$ which combinatorics is constant on a Zariski open set $\Zt$ in the 
space of generic arrangements $H^0_i,i=1,...,n$ 
but does not describe the set $\Zt$ explicitly.
In 1994 (see \cite{falk}) Falk showed
that, contrary to what was frequently 
stated (see for instance \cite{Or}, sect. 8, \cite{OT} or \cite{RL}), the combinatorial type of $\B(n,k)$ depends on 
the arrangement $\A$ of hyperplanes 
$H^0_i,i=1,...,n$  by providing an example of $\A$ 
for  which the corresponding discriminantal arrangement has combinatorial 
type distinct from the one which occurs when $\A$
varies within the Zariski open set $\Zt$. In 1997 Bayer and Brandt ( cf. \cite{BB} ) called the arrangements $\A$ in $\Zt$ \textit{very generic} and conjectured the full description of intersection lattice of $\B(n,k)$ if $\A \in \Zt$. In 1999 Athanasiadis proved their conjecture (cf. \cite{athana}). In particular, for the case of arrangement $\A$ in $\R^k$, endowed with standard metric, he introduced a degree $m$ polynomial $p_{\mathbb{T}}(a_{ij})$ ( section 1 in \cite{athana} and subsection \ref{plucker} in this paper ) in the indeterminates  $(a_{ij})$ where $\alpha_i=(a_{ij})$ is the normal vector to hyperplane $H_i^0$, $i \in L_h \in \mathbb{T}$, $L_h$ is a subset of cardinality $k+1$ of $\{1,\ldots ,n\}$ and $\mathbb{T}$ is a set of cardinality $m$. Since null space of this polynomial corresponds to intersection of hyperplanes in $\B(n,k)$, he provided, in case of very generic arrangements, a full description of sets $\mathbb{T}$ such that $p_{\mathbb{T}}(a_{ij})=0$ (cf. Theorem 3.2 in \cite{athana}). In particular all codimension $2$ intersections of hyperplanes in $\B(n,k)$ have multiplicity $2$ or $k+2$ if $\A$ is very generic.

More recently, in 2016 ( cf. \cite{sette}), Libgober and second author gave a sufficient geometric condition for an arrangement $\A$ not to be very generic. In particular they gave a necessary and sufficient condition for multiplicity $3$ codimension $2$ intersections of hyperplanes in $\B(n,k)$ to appear ( Theorem 3.8 \cite{sette} and Theorem \ref{thm:sette} in this paper).

The purpose of this short note is double. From one side to rewrite the result obtained in \cite{sette} in terms of the polynomial $p_{\mathbb{T}}(a_{ij})$ introduced by Athanasiadis and prove that, in case of non very generic arrangements, if $\mathbb{T}$ is a set of cardinality $3$ such that  $p_{\mathbb{T}}(a_{ij})=0$, then polynomial $p_{\mathbb{T}}(a_{ij})$ has a simpler polynomial expression $\widetilde{p}_\mathbb{T}(a_{ij})$.

On the other side to show, by mean of a more algebraic point of view, that non very generic arrangements $\A$ of cardinality $n$ in $\CC^3$ are points in a well defined degree $2$ hypersurface in the projective Grassmannian $Gr(3,n)$. Indeed the space of generic arrangements of $n$ lines in $(\PP^2)^n$ is Zariski open set $U$ in the space of all arrangements of $n$ lines in $(\PP^2)^n$. On the other hand in $Gr(3,n)$ there is open set $U'$ consisting of $3$-spaces intersecting each coordinate hyperplane transversally (i.e. having dimension of intersection equal $2$). One has also one set $\tilde U$ in $Hom(\CC^3,\CC^n)$ consisting of embeddings with image transversal to coordinate hyperplanes and $\tilde U/GL(3)=U'$ and $\tilde U/(\CC^*)^n=U$. 
Hence generic arrangements can be regarded as points in $Gr(3,n)$. 

The content of paper is the following.

In section \ref{pre}  we recall definition of discriminantal arrangement from \cite{man}, basic results in \cite{sette}, definition of $p_{\mathbb{T}}(a_{ij})$ in \cite{athana} and basic notions on Grassmannian ( cf. \cite{harris} ). In section \ref{thmU63} we give a full description of main example $\B(6,3)$ of $6$ hyperplanes in $\RR^3$. Section \ref{thmUnk} contains the result stating  equivalence of polynomial $p_{\mathbb{T}}(a_{ij})$ with its reduced form $\widetilde{p}_{\mathbb{T}}(a_{ij})$ (cf. Theorem \ref{thm:Bnk}). The last section contain the last result of this paper (cf. Theorem \ref{thm:gr}) describing a family of hypersurfaces in projective Grassmannian $Gr(3,n)$ in terms of non very generic arrangements $\A$ in $\CC^3$. Notice that in Sections \ref{thmU63} and  \ref{thmUnk}  $\A$ is an arrangement in $\R^k$ while in Sections \ref{grassmannian} $\A$ is an arrangement in $\C^k$.

Finally, authors want to thank A. Libgober and an anonymous referee for very useful comments.

\section{Preliminaries}\label{pre}

\subsection{Discriminantal arrangement}\label{discarr}

Let $H^0_i, i=1,...,n$ be a generic arrangement in $\CC^k, k<n$ i.e. 
a collection of hyperplanes such that $\dim \bigcap_{i \in K,
 \mid K\mid=k} H_i^0=0$. 
Space of parallel translates $\SS(H_1^0,...,H_n^0)$ (or simply $\SS$ when 
dependence on $H_i^0$ is clear or not essential)
is the space of $n$-tuples
$H_1,...,H_n$ such that either $H_i \cap H_i^0=\emptyset$ or 
$H_i=H_i^0$ for any $i=1,...,n$.
One can identify $\SS$ with $n$-dimensional affine space $\CC^n$ in
such a way that $(H^0_1,...,H^0_n)$ corresponds to the origin.  In particular, an ordering of hyperplanes in $\A$ determines the coordinate system in $\SS$ (see \cite{sette}).

We will use the compactification of $\CC^k$ viewing it 
as $\PP^k\setminus H_{\infty}$ endowed with collection of hyperplanes
$\bar H^0_i$ which are projective closures of affine hyperplanes
$H^0_i$. Condition of genericity is equivalent to $\bigcup_i H^0_i$ 
being a normal crossing divisor in $\PP^k$.

For a generic arrangement $\A$ in $\CC^k$ 
formed by hyperplanes $H_i, i=1,...,n$
{\it the trace at infinity} (denoted by $\A_{\infty}$)  is the arrangement 
formed by hyperplanes 
$H_{\infty,i}=\bar H^0_i\cap H_{\infty}$.  
  
The trace  $\mathcal{A}_{\infty}$ of an arrangement $\mathcal{A}$ determines the space of parallel translates $\mathbb{S}$ (as a subspace in the space of $n$-tuples of hyperplanes in $\mathbb{P}^k$). For a $t$-tuple $H_{i_1}, \dots, H_{i_t}$ ($t \geq 1$) of hyperplanes in $\mathcal{A}$, recall that the arrangement which is obtained by intersections of hyperplanes $H \in \mathcal{A}, H \neq H_{i_s}$, $s = 1, \dots, t$ with $H_{i_1} \cap \dots \cap H_{i_t}$, is called the \textit{restriction} of $\mathcal{A}$ to $H_{i_1} \cap \dots \cap H_{i_t}$.

For a generic arrangement  $\mathcal{A}_{\infty}$, consider the closed subset of $\mathbb{S}$ formed by those collections which fail to form a generic arrangement. This subset is a union of hyperplanes with each hyperplane $D_L$ corresponding to a subset $L = \{ i_1, \dots, i_{k+1} \} \subset  [n]:= \{ 1, \dots, n \}$ and consisting of $n$-tuples of translates of hyperplanes $H_1^0, \dots, H_n^0$ in which translates of $H_{i_1}^0, \dots, H_{i_{k+1}}^0$ fail to form a generic arrangement. The arrangement $\mathcal{B}(n, k, \mathcal{A}_{\infty})$ of hyperplanes $D_L$ is called $the$ $discriminantal$ $arrangement$ and has been introduced by Mannin and Schectman (see \cite{man}).  Notice that since combinatorics of discriminantal arrangement depends on the arrangement $\mathcal{A}_{\infty}$ rather than $\A$, we denote it by $\mathcal{B}(n, k, \mathcal{A}_{\infty})$ following notation in \cite{sette}.
%In this case we will simply write $\mathcal{B}(n,k)$. 

\subsection{Good 3s-partitions}
Given $s \geq 2$ and $n \geq 3s$, consider the set $\mathbb{T} = \{ L_1, L_2, L_3 \}$, with $L_i$ subsets of $[n]$ such that $|L_i| = 2s$, $|L_i \cap L_j| = s$ ($i \neq j$), $L_1 \cap L_2 \cap L_3 = \emptyset$ (in particular $|\bigcup L_i| = 3s$) with a choice $L_1 = \{ i_1, \dots, i_{2s}\}, L_2 = \{ i_{s+1}, \dots, i_{3s} \}, L_3 = \{ i_1,\dots, i_s, i_{2s+1}, \dots, i_{3s} \}$. We call the set $\mathbb{T} = \{ L_1, L_2, L_3 \}$ \textit{a good $3$s-partition}.

Given a generic arrangement $\mathcal{A} $ in $\mathbb{C}^k$, subsets $L_i$ define hyperplanes $D_{L_i}$ in the discriminantal arrangement $\mathcal{B}(n, k, \mathcal{A}_{\infty})$. In the rest of the paper we will always use $D_{L}$ to denote hyperplanes in discriminantal arrangement. 
With above notations the following lemma holds.

\begin{lem}{(Lemma 3.1 \cite{sette})}\label{lem:sette}  Let $s \geq$ 2, $n = 3s$, $k = 2s-1$ and $\mathcal{A}$ be a generic arrangement of $n$ hyperplanes in $\mathbb{C}^k$. Given a good  
$3s$-partition $\mathbb{T} = \{ L_1, L_2, L_3 \}$ of [$n$] = [$3s$], consider the triple of codimension $s$ subspaces $H_{\infty, i, j} = \bigcap_{t \in L_i \cap L_j} H_{\infty,t}$ of the hyperplane at infinity $H_{\infty}$. Then $H_{\infty, i, j}$ span a proper subspace in $H_{\infty}$ if and only if the codimension of $D_{L_1} \cap D_{L_2} \cap D_{L_3}$ is 2. 
\end{lem}

In \cite{sette} authors define a notion of dependency for a generic arrangement $\mathcal{A}_\infty = \{ W_{\infty, 1}, \dots , W_{\infty, 3s} \}$ in $\mathbb{P}^{2s-2}, s \geq 2$ based on Lemma \ref{lem:sette} as follows. If there exists a partition $I_{1}, I_{2}$ and $I_{3}$ of [$3s$] such that $P_i = \bigcap_{t\in I_i} W_{\infty, t}$ span a proper subspace in $\mathbb{P}^{2s-2}$, then $\mathcal{A}_\infty$ is called \textit{dependent}. Remark that if $\{ L_1, L_2, L_3 \}$ is a good $3s$-partition and we set $I_1 = L_1 \cap L_2$, $I_2 = L_1 \cap L_3$, $I_3 = L_2 \cap L_3$ then the the assumption of Lemma \ref{lem:sette} is that the trace at infinity $\mathcal{A}_\infty$ of $\A$ is dependent  and the following theorem holds .

\begin{thm}(Theorem 3.8 \cite{sette}) \label{thm:sette}
Let $\mathcal{A}$ be a generic arrangement of $n$ hyperplanes in $\mathbb{C}^k$ and $\mathcal{A}_\infty$ the trace at infinity of $\mathcal{A}$. 

1. The arrangement $\mathcal{B}(n,k,\mathcal{A}_\infty)$ has $\binom{n}{k+2}$ codimension 2 strata of multiplicity $k+2$.

2. 
There is one to one correspondence between 

(a) restriction arrangements of $\mathcal{A}_\infty$ which are dependent, and

(b) triples of hyperplanes in $\mathcal{B}(n,k,\mathcal{A}_\infty)$ for which the codimension of their intersection is equal to 2.

3. There are no codimension 2 strata having multiplicity 4 unless $k=3$. All codimension 2 strata of $\mathcal{B}(n,k,\mathcal{A}_\infty)$ not 
mentioned  in part 1, have multiplicity either 2 or 3.

4. Combinatorial type of $\mathcal{B}(n,2,\mathcal{A}_\infty)$ is independent of $\mathcal{A}$.
\end{thm}

%%%%%%%%%%%%%%%%%%%%%%%%%%%%%%%%%%%%%%%%%%%%%%%%%%%%%%%%%%%%%%%%%%%

\subsection{Matrices $A(\mathcal{A}_\infty)$ and $A_{\mathbb{T}}(\mathcal{A}_\infty)$ }\label{plucker}
Let $\alpha_i = (a_{i1}, \dots , a_{ik}$) be  the normal vectors of hyperplanes $H^0_i$, $1 \leq i \leq n$, in the generic arrangement $\mathcal{A}$ in $\C^k$. Normal here is intended with respect to the usual dot product 
$$
(a_1, \ldots, a_k)\cdot (v_1,\ldots, v_k)=\sum_i a_iv_i \quad .
$$
Then the normal vectors to hyperplanes $D_L$, $L = \{s_1< \dots < s_{k+1} \} \subset$ [$n$] in $ \mathbb{S} \simeq \mathbb{C}^n$ are nonzero vectors of the form 
\begin{equation}\label{eq:normvec}
\alpha_L = \sum^{k+1}_{i=1} (-1)^i \det (\alpha_{s_1}, \dots, \hat{\alpha_{s_i}}, \dots, \alpha_{s_{k+1}})e_{s_i},
\end{equation} 
where $\{e_j\}_{1\leq j \leq n}$ is the standard basis of $\mathbb{C}^n$ (cf. \cite{athana}). 

Let $\mathcal{P}_{k+1}([n]) = \{L \subset [n] \mid |L| = k+1\}$ be the set of cardinality $k+1$ subsets of $[n]$, we denote by
\begin{equation}\label{An}
 A(\mathcal{A}_\infty) = (\alpha_L)_{L \in \mathcal{P}_{k+1} ([n])}
\end{equation}
the matrix having in each row the entries of vectors $\alpha_L$ normal to hyperplanes $D_{L}$ and by $A_{\mathbb{T}}(\mathcal{A}_\infty)$ the submatrix  of $A(\mathcal{A}_\infty)$ with rows $\alpha_L$, $L \in \mathbb{T}$, $\mathbb{T} \subset \mathcal{P}_{k+1}([n])$ of cardinality $m$.

\subsection{Polynomial $p_{\mathbb{T}}(a_{ij})$ }
Construction in Subsection \ref{plucker} naturally holds also in real case, i.e. $\A$ arrangement in $\RR^k$. In this case Athanasiadis (see \cite{athana} ) defined the polynomial
\begin{equation}\label{eq:athanapoly}
p_{\mathbb{T}}(a_{ij}) = \sum_{\substack{J \subset [n] \\ |J|=m}} \det [A_{\mathbb{T},J}(\mathcal{A}_\infty)]^2 \quad 
\end{equation} 
in the variable $a_{ij}$ given by the sum of the squares of determinants of the $m \times m$ submatrices $A_{\mathbb{T}, J}$ of $A_{\mathbb{T}}(\mathcal{A}_\infty)$ obtained considering columns $j \in J$. 
%In this paper we are mainly interested in case of $|\mathbb{T}|=3$. 
Notice that if $\A$ is a generic arrangement in $\RR^k$, if $\mathbb{T} = \{L_1, L_2, L_3\}$ is a good $3s$-partition then condition in Lemma \ref{lem:sette} is equivalent to $p_{\mathbb{T}}(a_{ij}) = 0$.

%%%%%%%%%%%%%%%%%%%%%%%%%%%%%%%%%%%%%%%%%%%%%%%%%%%%%%%%%%%%%%%%%%%%%%%%%%%%%
\subsection{Grassmannian $Gr(k,n)$}

Let $Gr(k, n)$ be the Grassmannian of $k$-dimensional subspaces of $\CC^n$ and 
\begin{equation}\label{eq:plemb}
\begin{split}
\gamma: Gr(k, n) &\to \mathbb{P}(\bigwedge^k \CC^n)  \\
<v_1, \dots, v_k> &\mapsto [v_1 \wedge \dots \wedge v_k],
\end{split}
\end{equation}
the Pl\"{u}cker embedding. Then $[x] \in \mathbb{P}(\bigwedge^k \CC^n)$ is in $\gamma(Gr(k,n))$ if and only if the map
\begin{equation}
\begin{split}
\varphi_x: \CC^n &\to \bigwedge^{k+1} \CC^n  \\ 
v &\mapsto v \wedge x
\end{split}
\end{equation} 
has kernel of dimension $k$, i.e. ker $\varphi_x = <v_1, \dots, v_k>$. If $e_1, \dots,e_n $ is a basis of $\CC^n$ then $e_I = e_{i_1}~\wedge~\dots~\wedge~e_{i_k}$, $I= \{ i_1, \dots, i_k \} \subset [n], i_1 < \dots < i_k$, is a basis for $\bigwedge^k \CC^n$ and $x \in \bigwedge^k \CC^n$ can be written uniquely as 

\begin{equation}\label{betaformula}
x = \displaystyle \sum_{\substack {I \subseteq [n] \\ |I| = k}} \beta_I e_I 
= \sum_{1 \leq i_1 < \dots < i_k \leq n} \beta_{i_1 \dots i_k} (e_{i_1} \wedge \dots \wedge e_{i_k}) 
\end{equation}
where homogeneous coordinates $\beta_I$ are the Pl\"{u}cker coordinates 
on $\mathbb{P}(\bigwedge^k \CC^n) = \mathbb{P}^{\binom nk-1}$ associated to 
the ordered basis $e_1, \dots, e_n$ of $\CC^n$. With this choice of basis for $\CC^n$ the matrix $M_x=(b_{ij})$ associated to $\varphi_x$ is the $\binom {n}{k+1} \times n$ matrix with rows indexed by ordered subsets $I \subseteq [n], |I| = k$, and entries $b_{ij}=(-1)^i \beta_{I\cup\{j\}\setminus \{i\}}$ if $i \in I$, $b_{ij}=0$ otherwise. Pl\"{u}cker relations, i.e conditions for dim(ker $\varphi_x$) = $k$, are vanishing conditions of all $(k+1) \times (k+1)$ minors of $M_x$. It is well known (see for instance \cite{harris}) that Pl\"{u}cker relations are degree 2 relations and they can also be written as 
\begin{equation}\label{pluck}
\sum_{l=0}^{k} (-1)^l \beta_{i_1 \dots i_{k-1}j_l} \beta_{j_0 \dots \hat{j_l} \dots j_k} =0
\end{equation}
for any $2k$-tuple $(i_1, \dots, i_{k-1}, j_0, \dots, j_k)$.

\begin{rem} Notice that vectors $\alpha_L$ in equation (\ref{eq:normvec}) normal to hyperplanes $D_L$  correspond to rows $I=L$ in the Pl\"{u}cker matrix $M_x$, that is $$A(\mathcal{A}_\infty)=M_x \quad .$$ 
For this reason in the rest of the paper we will call $A(\mathcal{A}_\infty)$  Pl\"{u}cker coordinate matrix.
Notice that, in particular, $\det (\alpha_{s_1}, \dots, \hat{\alpha_{s_i}}, \dots, \alpha_{s_{k+1}})$ is the Pl\"{u}cker coordinate $\beta_I , I = \{ s_1, s_2, \dots, s_{k+1} \} \backslash \{ s_i \}$.
\end{rem}
In the following section we give an example to illustrate the general Theorem in section \ref{thmUnk}. This example appears also in \cite{falk}, \cite{sette} and, in the context of oriented matroids, in \cite{BaKe}.

%%%%%%%%%%%%%%%%%%%%%%%%%%%%%%%%%%%%%%%%%%%%%%%%%%%%%%%%%%%%%%%%%%%%%%%%%%%%%

\section{Example $\mathcal{B}$(6, 3, $\mathcal{A}_\infty$) in real case} \label{thmU63}

Consider $\mathcal{A} = \{ H_1^0, H_2^0, \dots, H_6^0 \}$ be a generic arrangement of hyperplanes in $\mathbb{R}^3$ with normal vectors 
$\alpha_i = ( a_{i1}, a_{i2}, a_{i3} )$, $1 \leq i \leq 6$ and $H_i^{t_i}$ be hyperplane obtained by translating $H_i^0$ along direction $\alpha_i$, i.e., $H_i^{t_i} = H_i^0 + t_i \alpha_i$, $t_i \in \mathbb{R}$.
Let $\mathbb{T} = \{ L_1, L_2, L_3\}$ be the good $6$-partition with $L_1 = \{ 1, 2, 3, 4 \}, L_2 = \{ 1, 2, 5, 6 \}$ and $L_3 = \{ 3, 4, 5, 6 \} $,  then  
$$A_{\mathbb{T}} (\mathcal{A}_{\infty}) = \begin{pmatrix}
\alpha_{L_1}\\
\alpha_{L_2}\\
\alpha_{L_3}
\end{pmatrix}
=
\begin{pmatrix}
-\beta_{234} & \beta_{134} & -\beta_{124} & \beta_{123} & 0 & 0\\
-\beta_{256} & \beta_{156} & 0 & 0 & -\beta_{126} & \beta_{125}\\
0 & 0 & -\beta_{456} & \beta_{356} & -\beta_{346} & \beta_{345}
\end{pmatrix}
, \quad
\beta_{ijk} = \det \begin{pmatrix}
a_{i1} & a_{j1} & a_{k1} \\
a_{i2} & a_{j2} & a_{k2} \\
a_{i3} & a_{j3} & a_{k3} 
\end{pmatrix}
$$

is a submatrix of the Pl\"{u}cker coordinate matrix $A (\mathcal{A}_{\infty})$.\\
Let $\alpha_i \times \alpha_j$ be the cross product of $\alpha_i, \alpha_j$ corresponding to the direction orthogonal to both $\alpha_i$ and $\alpha_j$ and denote by $(\alpha_i \times \alpha_{i+1})$ the matrix $\begin{pmatrix}
\alpha_1 \times \alpha_2\\
\alpha_3 \times \alpha_4\\
\alpha_5 \times \alpha_6
\end{pmatrix}$. Then $\alpha_i \times \alpha_j$ is the direction of the line $H_i \cap H_j$, since $\alpha_i$ and $\alpha_j$ are, respectively, directions orthogonal to $H_i$ and $H_j$ and  $\rank A_{\mathbb{T}} (\mathcal{A}_{\infty})=2$ if and only  $\rank(\alpha_i \times \alpha_{i+1})=2$. Indeed $\rank(A_{\mathbb{T}} (\mathcal{A}_{\infty}))=2$ is equivalent to codim ($D_{L_1} \cap D_{L_2} \cap D_{L_3}$) = 2, hence by Lemma \ref{lem:sette}, the points  ${\displaystyle \bigcap_{i \in L_1 \cap L_2}} \bar H_i^{t_i} \cap H_\infty = \bar H_3^{t_3} \cap \bar H_4^{t_4} \cap H_\infty, {\displaystyle \bigcap_{i \in L_1 \cap L_3}} \bar H_i^{t_i} \cap H_\infty = \bar H_1^{t_1} \cap \bar H_2^{t_2} \cap H_\infty$, and ${\displaystyle \bigcap_{i \in L_2 \cap L_3}} \bar H_i^{t_i} \cap H_\infty = \bar H_5^{t_5} \cap \bar H_6^{t_6} \cap H_\infty$ are collinear, that is 
directions of $H_i^{t_i} \cap H_{i+1}^{t_{i+1}}$ are dependent and hence $\rank
(\alpha_i \times \alpha_{i+1})=2$ (see Fig. \ref{fig1} ). 

\begin{figure}%[htbp]
 \begin{center}
 % \fbox{
  \includegraphics[width=80mm]{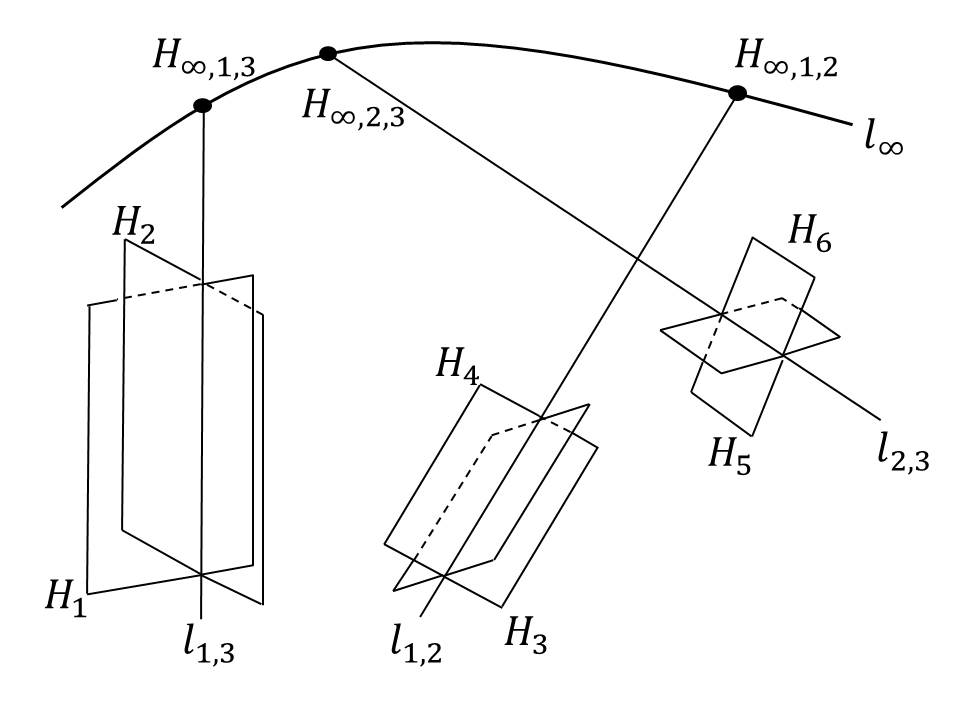}
  %}
 \end{center}
 \caption{Picture of case $\mathcal{B}(6, 3, \mathcal{A}_{\infty}^0)$}
 \label{fig1}
\end{figure}

%\begin{rem}\label{rem:rem2}
Rank of $A_\mathbb{T}(\mathcal{A}_\infty)$ is equal to 2 if and only if $\beta_{ijk}$  are solutions of the system:

\begin{equation}\label{rel:beta3}
 \begin{aligned}[c]
(I)\begin{cases}
   -\beta_{456}( \beta_{134}\beta_{256} - \beta_{234}\beta_{156}) = 0 \\
    \beta_{356}( \beta_{134}\beta_{256} - \beta_{234}\beta_{156}) = 0 \\
   -\beta_{346}( \beta_{134}\beta_{256} - \beta_{234}\beta_{156}) = 0 \\
    \beta_{345}( \beta_{134}\beta_{256} - \beta_{234}\beta_{156}) = 0 \\
   -\beta_{256}( \beta_{124}\beta_{356} - \beta_{123}\beta_{456}) = 0 \\
   \beta_{156}( \beta_{124}\beta_{356} - \beta_{123}\beta_{456}) = 0 \\
   -\beta_{126}( \beta_{124}\beta_{356} - \beta_{123}\beta_{456}) = 0 \\
    \beta_{125}( \beta_{124}\beta_{356} - \beta_{123}\beta_{456}) = 0 \\
    -\beta_{234}( \beta_{125}\beta_{346} - \beta_{126}\beta_{345}) = 0\\
      \beta_{134}( \beta_{125}\beta_{346} - \beta_{126}\beta_{345}) = 0 \\
   -\beta_{124}( \beta_{125}\beta_{346} - \beta_{126}\beta_{345}) = 0 \\
    \beta_{123}( \beta_{125}\beta_{346} - \beta_{126}\beta_{345}) = 0   
    \end{cases}
\end{aligned}
\qquad \mbox{ and} \qquad
 \begin{aligned}[c]
 (II) \begin{cases}
  \beta_{234}\beta_{126}\beta_{456} + \beta_{124}\beta_{256}\beta_{346} = 0 \\
   -(\beta_{234}\beta_{125}\beta_{456} + \beta_{124}\beta_{256}\beta_{345}) = 0 \\
   -(\beta_{234}\beta_{126}\beta_{356} + \beta_{123}\beta_{256}\beta_{346}) = 0 \\
     \beta_{234}\beta_{125}\beta_{356} + \beta_{123}\beta_{256}\beta_{345} = 0 \\    
  -(\beta_{134}\beta_{126}\beta_{456} + \beta_{124}\beta_{156}\beta_{346}) = 0 \\
    \beta_{134}\beta_{125}\beta_{456} + \beta_{124}\beta_{156}\beta_{345} = 0 \\
    \beta_{134}\beta_{126}\beta_{356} + \beta_{123}\beta_{156}\beta_{346} = 0 \\
  -(\beta_{134}\beta_{125}\beta_{356} + \beta_{123}\beta_{156}\beta_{345}) = 0 \\
    \end{cases} \quad 
  \end{aligned}
\end{equation}

and polynomial $p_{\mathbb{T}}(a_{ij})$ is 

\bigskip

$p_{\mathbb{T}}(a_{ij}) = {\displaystyle \sum_{\substack{J \subset [6] \\ |J|=3}}} \det (A_{\mathbb{T}, J})^2 =$ 
$(\beta_{134}\beta_{256} - \beta_{234}\beta_{156})^2 ({\displaystyle \sum_{\substack {I_1 \subset \{3, 4, 5, 6 \} \\ |I_1| =3}}} \beta_{I_1}^2)$
+
$(\beta_{124}\beta_{356} - \beta_{123}\beta_{456})^2 ({\displaystyle \sum_{\substack {I_2 \subset \{1, 2, 5, 6 \} \\ |I_2| =3}}} \beta_{I_2}^2)$
+
$(\beta_{125}\beta_{346} - \beta_{126}\beta_{345})^2 ({\displaystyle \sum_{\substack {I_3 \subset \{1, 2, 3, 4 \} \\ |I_3| =3}}} \beta_{I_3}^2)$
+
${\displaystyle \sum_{\substack {i = 5, 6 \\ j = 3, 4}}} (\beta_{234} \beta_{12i} \beta_{j56} + \beta_{12j} \beta_{256} \beta_{34i})^2$
+
${\displaystyle \sum_{\substack {i = 5, 6 \\ j = 3, 4}}} (\beta_{134} \beta_{12i} \beta_{j56} + \beta_{12j} \beta_{156} \beta_{34i})^2$.

\bigskip

On the other hand the condition $\rank(\alpha_i \times \alpha_{i+1})=2$
is simply
$\det (\alpha_i \times \alpha_{i+1})=0$ and if we define 
\begin{equation}
\widetilde{p}_\mathbb{T}(a_{ij}) = [ \det (\alpha_i \times \alpha_{i+1}) ]^2
= \{ (a_{12}a_{23} - a_{13}a_{22}) \Delta_{11} +  (a_{11}a_{23} - a_{13}a_{21}) \Delta_{12} +  (a_{11}a_{22} - a_{12}a_{2}) \Delta_{13} \}^2,
\end{equation}
$\Delta_{1l}$ cofactors of $(\alpha_i \times \alpha_{i+1})$, then  $p_{\mathbb{T}}(a_{ij}) = 0$ if and only if $\widetilde{p}_\mathbb{T}(a_{ij}) = 0$. That is polynomial $\widetilde{p}_\mathbb{T}(a_{ij})$ is a polynomial of, in general, lower degree than $p_\mathbb{T}(a_{ij})$ with the same set of zeros.

%%%%%%%%%%%%%%%%%%%%%%%%%%%%%%%%%%%%%%%%%%%%%%%%%%%%%%%%%%%%%%%%

\section{Polynomial $\widetilde{p}_\mathbb{T}(a_{ij})$ in $\mathcal{B}(n,k, \mathcal{A}_{\infty})$ in real case}\label{thmUnk}

\subsection{Case $\mathcal{B}$($n$, 3, $\mathcal{A}_\infty$)}\label{ncase}
 It is straightforward to generalize the example in section \ref{thmU63} to the case of $n$ hyperplanes in $\mathbb{R}^3$. Denote by $(\alpha_{i_j} \times \alpha_{i_{j+1}})$ the matrix $\begin{pmatrix}
\alpha_{i_1} \times \alpha_{i_2}\\
\alpha_{i_3} \times \alpha_{i_4}\\
\alpha_{i_5} \times \alpha_{i_6}  
\end{pmatrix}$, the following Theorem holds.

\begin{thm}\label{thm:Bn3}
Let $\mathcal{A}$ be a generic arrangement of $n$ hyperplanes in $\mathbb{R}^3$ with normal vectors 
$\alpha_i = (a_{i1}, a_{i2}, a_{i3})$. Let $\mathbb{T} = \{ L_1, L_2, L_3\}$ be a good 6-partition with a choice $L_1 = \{ i_1, i_2, i_3, i_4\}, L_2 = \{ i_3, i_4, i_5, i_6\}$ and $L_3 = \{ i_1, i_2, i_5, i_6 \}$ and $A_\mathbb{T}(\mathcal{A_\infty})$ be the matrix with rows $\alpha_{L_1}$, $\alpha_{L_2}$, $\alpha_{L_3}$. Then the followings are equivalent:
\begin{enumerate}
\item $\rank A_\mathbb{T}(\mathcal{A}_\infty)=2$;
\item $p_\mathbb{T}(a_{ij}) = 0$;
\item $\rank (\alpha_{i_j} \times \alpha_{i_{j+1}})=2$;
\item $\widetilde{p}_\mathbb{T}(a_{ij}) = [\det (\alpha_{i_j} \times \alpha_{i_{j+1}})]^2=0$.
\end{enumerate} 
\end{thm}

\begin{proof}
The equivalences (1) $\Leftrightarrow$ (2) and (3) $\Leftrightarrow$ (4) are obvious by definitions of $p_\mathbb{T}(a_{ij})$ and $\widetilde{p}_\mathbb{T}(a_{ij})$. The proof that (1) $\Leftrightarrow$ (3) can be obtained by remarks in Section\ref{thmU63} relabeling indices $1, \dots, 6$ with $i_1, \dots, i_6$.
\end{proof}

\begin{rem} Notice that since $\widetilde{p}_\mathbb{T}(a_{ij}) = [\det (\alpha_{i_j} \times \alpha_{i_{j+1}})]^2$, then $\widetilde{p}_\mathbb{T}(a_{ij}) =0$ if and only if $\det (\alpha_{i_j} \times \alpha_{i_{j+1}})=0$ and equivalence of conditions $(1), (3) $ and $(4)$ in Theorem \ref{ncase} holds also for generic arrangements in $\CC^3$.
\end{rem}

%%%%%%%%%%%%%%%%%%%%%%%%%%%%%%%%%%%%%%%%%%%%%%%%%%%%%%%%%%%%%%%%%%%%%%%%%%%

\subsection{Generalization to $\mathcal{B}(n,k, \mathcal{A}_{\infty})$}
Let $\mathcal{A} = \{ H_1, \dots, H_n \}$ be a generic arrangement of hyperplanes in $\mathbb{R}^k$ and $\mathbb{T} = \{ L_1, L_2, L_3 \}$ be a good $3s$-partition of indices in [$n$]. If $\alpha_\tau$ are normal vectors to $H_\tau \in \mathcal{A}$, $\tau = 1, \dots, n$, $T = \{ j_1, \cdots, j_t \}$ a subset of [$n$] which has empty intersection with $L_1 \cup L_2 \cup L_3$, define vector spaces
$$U_{i,j}^{\perp} = \{ v \in \mathbb{R}^k \ | \ v \cdot \alpha_\tau = 0, \tau \in L_i \cap L_j \},$$

where $v \cdot \alpha_\tau$ is the scalar product of $v$ and $\alpha_\tau$, and  
\begin{equation}\label{eq:W_T}
   W_T 
   = \begin{cases}
  	  \mathbb{R}^k & \text{($T = \emptyset$)} \\
   	 \{ v \in \mathbb{R}^k \ | \ v \cdot \alpha_\tau = 0,  \tau \in T \} & \text{($T \neq \emptyset$)} \qquad . 
     \end{cases}
\end{equation}
Then $W_T$ is the vector space associated to ${\displaystyle \bigcap_{\tau \in T}} H_\tau$ and $U_{i,j}^{\perp} \cap W_T = \{ v \in \mathbb{R}^k \ | \ v \cdot \alpha_\tau = 0, \tau \in (L_i \cap L_j) \cup T \}$ is a vector space of dimension $k - (s + t) $, where $s$ and $t$ are, respectively, cardinalities of $L_i \cap L_j$ and $T$. With above notations, define the polynomial
\begin{center}
$\widetilde{p}_{\mathbb{T},T}(a_{ij}) = {\displaystyle \sum_{\substack U \in \mathbb{U}_{\mathbb{T},T}}} [ \det U ]^2$,
\end{center}
where $\mathbb{U}_{\mathbb{T},T}$ is the set of all $k \times k$ submatrices of the $3(k - s - t) \times k$ matrix having as rows the vectors spanning $U_{i,j}^{\perp} \cap W_T$. 

If $k = 2s - 1$ and $n = 3s$, $s \geq 2$, we have $T = \emptyset$ and hence $U_{i,j}^{\perp} \cap W_T=U_{i,j}^{\perp}$ is a space of dimension $\dim U_{i,j}^{\perp}=s - 1$. $\mathbb{U}_{\mathbb{T}, \emptyset}$ is the set of all $(2s-1) \times (2s-1)$ submatrices of the $3(s - 1) \times (2s-1)$ matrix having as rows the vectors spanning $U_{i,j}^{\perp}$ and the following lemma equivalent to Lemma \ref{lem:sette} holds.

\begin{lem} Let $s \geq 2$, $n=3s$, $k= 2s - 1$, i.e. $T = \emptyset$, and $\A$ be a generic arrangement of $n$ hyperplanes in $\RR^k$. Given a good $3s$-partition $\mathbb{T} = \{ L_1, L_2, L_3 \}$ of $[3s]=[n]$, $U_{i,j}^{\perp}$ span a proper subspace of $\mathbb{R}^{k}$ if and only if the rank of $A_\mathbb{T}(\mathcal{A}_{\infty})$ is 2. That is $\widetilde{p}_{\mathbb{T}, \emptyset}(a_{ij}) = 0$ if and only if $p_\mathbb{T}(a_{ij}) = 0$.
\end{lem}

\begin{proof} Since $\mathbb{T}$ is a good $3s$-partition and $A_\mathbb{T}(\mathcal{A}_{\infty}) = (\alpha_L)_{L \in \mathbb{T}}$ is a $3 \times n$ matrix, the rank of the matrix $A_\mathbb{T}(\mathcal{A}_{\infty})$ is equal to 2 if and only if $\alpha_{L}, L \in \mathbb{T}$, are linearly dependent that is the intersection $D_{L_1} \cap D_{L_2} \cap D_{L_3}$ of hyperplanes in $\mathcal{B}(n, k, \mathcal{A}_{\infty})$ is a space of codimension 2. Then by Lemma \ref{lem:sette} this corresponds to $H_{\infty, i, j} =   \bigcap_{\tau \in L_i \cap L_j} \bar H_\tau \cap H_\infty$ $\subset H_{\infty}$ span a proper subspace in $H_{\infty}$. Let $V_{\tau}$ be the associated vector spaces to hyperplanes $H_{\tau}$, hence $V_{i,j} =  \bigcap_{\tau \in L_i \cap L_j} V_\tau$ are the associated vector spaces to $H_{i,j} =  \bigcap_{\tau \in L_i \cap L_j} H_\tau$ and $V_{i,j} = U_{i,j}^{\perp}$ since $v \in V_{i,j}$ if and only if $v \cdot \alpha_{\tau} = 0$ for any $\tau \in L_i \cap L_j$. It follows that $H_{\infty, i, j}$ span a proper subspace of $H_{\infty}$ if and only if $U_{i,j}^{\perp}$ span a proper subspace of $\mathbb{R}^k$. That is $\det U =0$ for any $U \in \mathbb{U}_{\mathbb{T}, \emptyset}$ or, equivalently, $\widetilde{p} _{\mathbb{T}, \emptyset} (a_{ij}) = 0$. 
\end{proof}
Notice that if $s$ = 2, i.e. case $\mathcal{B}$(6, 3, $\mathcal{A}_\infty$), $\widetilde{p}_{\mathbb{T}, \emptyset}(a_{ij})$ coincides with $\widetilde{p}_{\mathbb{T}}(a_{ij})$ defined in Section \ref{thmU63}. In this case 1-dimensional subspaces $U_{1,2}^{\perp}$, $U_{1,3}^{\perp}$ and $U_{2,3}^{\perp}$ are spanned, respectively, by $\alpha_{1} \times \alpha_{2}, \alpha_{3} \times \alpha_{4}$ and $\alpha_{5} \times \alpha_{6}$, that is they are the lines drown in Figure \ref{fig1}.\\
Analogously to \cite{sette} we call a generic arrangement $\mathcal{A}$ = $\{ W_1,  \cdots, W_{3s} \}$ in $\mathbb{R}^{2s-1}$, $s \geq 2$, dependent if there exists a good 3$s$-partition such that  $U_{i,j}^{\perp}$ span a proper subspace of $\mathbb{R}^{2s-1}$. With this notation, by Lemma \ref{lem:sette} and Theorem \ref{thm:sette}, the following theorem holds.

\begin{thm}\label{thm:Bnk} Let $\A$ be a generic arrangement of $n$ hyperplanes in $\mathbb{R}^k$, $\mathbb{T}$ a good $3s$-partition, $3s \leq n$, and $T=[n] \setminus \cup_{L \in \mathbb{T}}L$. If $W_T$ is the vector space defined in equation (\ref{eq:W_T}), then rank of $A_{\mathbb{T}} (\mathcal{A_\infty})$ is equal to $2$ if and only if the restriction arrangement
$$\mathcal{A}_{W_T}= \{ H \cap {\displaystyle \bigcap_{\tau \in T}} H_\tau  \ | \ H \in \mathcal{A} \backslash \{H_i\}_{i \in T} \} $$
is dependent. With this choice of $T$ and $\mathbb{T}$ we get that $p_{\mathbb{T}}(a_{ij}) =0$ if and only if $\widetilde{p}_{\mathbb{T}, T} (a_{ij}) =0.$
\end{thm}

\begin{rem} 
%(Condition on vanishing of polynomial $p_\mathbb{T}(a_{ij})$)
For a fixed good $3s$-partition $\mathbb{T}$, equation $p_\mathbb{T}(a_{ij})=0$ corresponds to $n \choose 3s$ $3s\choose s$ nonlinear relations on Pl\"ucker coordinates $\beta _I$, $(2s-1) \times (2s-1)$ minors of the matrix $A=(a_{ij})$. \\
On the other hand $\widetilde{p}_{\mathbb{T}, T}(a_{ij}) =0$  is equivalent to vanishing of $(2s-1) \times (2s-1)$ minors of the matrix with rows given by solutions of system $A_I \cdot x =0$, $A_{I} = (a_{ij})_{i \in I}$, i.e. 
$n \choose 3s$ $3s-3 \choose 2s-1$ equations on $a_{ij}$.  That is $\widetilde{p}_{\mathbb{T}, T}(a_{ij}) =0$ is reduced form of $p_\mathbb{T}(a_{ij})=0$.    
\end{rem}

%%%%%%%%%%%%%%%%%%%%%%%%%%%%%%%%%%%%%%%%%%%%%%%%%%%%%%%%%%%%%%%%%%%%%%%

\section{Hypersurfaces in complex Grassmannian $Gr(3,n)$}\label{grassmannian}

Let now $\A$ be a generic arrangement of  $6$ hyperplanes in $\mathbb{C}^3$ (i.e. the example in Section \ref{thmU63} in $\mathbb{C}^3$ instead of $\mathbb{R}^3$) and
\begin{equation}
A=
\begin{pmatrix}
a_{11} & a_{12} & a_{13} \\
\vdots & \vdots & \vdots \\
a_{61} & a_{62} & a_{63}
\end{pmatrix}
\end{equation}
be the matrix having in each row normal vectors $\alpha_i$ to hyperplanes $H_i^0 \in \A$. Since $\A$ is generic, columns of $A$ are independent vectors in $\CC^6$ and they span a subspace of dimension $3$ in $\mathbb{C}^6$, i.e. an element in the Grassmannian $Gr(3,6)$. The non null $3 \times 3$ minors of $A$ are Pl\"{u}cker coordinates  $\beta_{ijk}$ and the matrix $A(\mathcal{A}_\infty)$ is the matrix of the map
\begin{equation*}
\begin{split}
\varphi_x : \mathbb{C}^6 &\to \bigwedge^4 \mathbb{C}^6 \\
v & \mapsto v \wedge x,
\end{split}
\end{equation*}
where 
$x = \sum_{1 \leq i < j < k \leq n} \beta_{i j k} (e_{i} \wedge e_{j} \wedge e_{k})$. If $\A_\infty$ is dependent then $\beta_{ijk}$ have to satisfy both, classical Pl\"{u}cker relations and relations in equation (\ref{rel:beta3}) (notice that since relations in equation (\ref{rel:beta3}) come directly from condition $\rank A_\mathbb{T}(\mathcal{A}_\infty)=2$, we get exactly same relations in real and complex case) . The latter can be simplified as:
\begin{center}
$(I): \,
  \begin{cases}
   (a) : \beta_{134}\beta_{256} - \beta_{234}\beta_{156} = 0 \\
   (b) : \beta_{124}\beta_{356} - \beta_{123}\beta_{456} = 0 \\
   (c) : \beta_{125}\beta_{346} - \beta_{126}\beta_{345} = 0
  \end{cases}
$
and \, \, \,
$(II): \,
  \begin{cases}
   (d) : \beta_{234}\beta_{126}\beta_{456} + \beta_{124}\beta_{256}\beta_{346} = 0 \\
   (e) : \beta_{234}\beta_{125}\beta_{456} + \beta_{124}\beta_{256}\beta_{345} = 0 \\
   (f) : \beta_{234}\beta_{126}\beta_{356} + \beta_{123}\beta_{256}\beta_{346} = 0 \\
   (g) : \beta_{234}\beta_{125}\beta_{356} + \beta_{123}\beta_{256}\beta_{345} = 0 \\
   (h) : \beta_{134}\beta_{126}\beta_{456} + \beta_{124}\beta_{156}\beta_{346} = 0 \\
   (i) : \beta_{134}\beta_{125}\beta_{456} + \beta_{124}\beta_{156}\beta_{345} = 0 \\
   (j) : \beta_{134}\beta_{126}\beta_{356} + \beta_{123}\beta_{156}\beta_{346} = 0 \\
   (k) : \beta_{134}\beta_{125}\beta_{356} + \beta_{123}\beta_{156}\beta_{345} = 0  \quad .
  \end{cases}
$
\end{center}

\bigskip
\noindent
Where equation $(I) (a)$ is obtained dividing the first four equations in system $(I)$ in (\ref{rel:beta3}) respectively by $-\beta_{456}, \beta_{356}, -\beta_{346}, \beta_{345}\neq 0$ and, similarly, equations $(I) (b)$ and $(c)$ are obtained dividing, respectively, equations from $5$ to $8$ and equations from $9$ to $12$  in system $(I)$ in (\ref{rel:beta3}) by, respectively, $-\beta_{256}, \beta_{156}, -\beta_{126}, \beta_{125}\neq 0$ and $-\beta_{234}, \beta_{134}, -\beta_{124}, \beta_{123}\neq 0$.  While equations in $(II)$ (\ref{rel:beta3}) are left unchanged except for a change of sign. Remark that this is only possible since $\A$ is a generic arrangement which implies that all $\beta_{ijk} \neq 0$ and hence we can divide equations in (\ref{rel:beta3}) $(I)$ opportunely by them. 
In the following we refer to equations in $(I)$ and $(II)$ by using corresponding letters, for example $(a)$ will refers to equation $\beta_{134}\beta_{256} - \beta_{234}\beta_{156}$.  Pl\"{u}cker relations in equation (\ref{pluck}) for $k=3$ becomes:
$$ \beta_{i_1 i_2 k_0}\beta_{k_1 k_2 k_3} - \beta_{i_1 i_2 k_1}\beta_{k_0 k_2 k_3} + \beta_{i_1 i_2 k_2}\beta_{k_0 k_1 k_3} - \beta_{i_1 i_2 k_3}\beta_{k_0 k_1 k_2} = 0 .$$
Fixing $i_1 =1, i_2=2, k_0=4, k_1=3, k_2=5, k_3=6$, we obtain $$\beta_{124}\beta_{356} - \beta_{123}\beta_{456} + \beta_{125}\beta_{436} - \beta_{126}\beta_{435} = 0, $$
that is $(b) = (c)$, and fixing $i_1 =5, i_2=6, k_0=2, k_1=1, k_2=3, k_3=4$ we get $(a) = (b)$. That is relations in (I) are equivalent.\\
Next we focus on type (II) relations and vanishing of all $4 \times 4$ minors of Pl\"{u}cker matrix . Fixed a good 6-partition $\mathbb{T} = \{L_1, L_2, L_3\}$, for any subset $L_4 \subset [6]$ of cardinality $4$ such that $L_4 \notin \mathbb{T}$, define the submatrix 
\begin{equation}
Pl_{\mathbb{T}}(D_{L_4}) = ( \alpha_{L_i} )_{1 \leq i \le 4} .
\end{equation} 
of $A(\mathcal{A}_\infty)$. The matrix $Pl_{\mathbb{T}}(D_{L_4})$ is obtained adding one row to the matrix $A_\mathbb{T}(\mathcal{A}_\infty)$. Hence since relations in equation (\ref{rel:beta3})
correspond to vanishing of $3 \times 3$ minors of $A_\mathbb{T}(\mathcal{A}_\infty)$, $\mathbb{T}=\{\{ 1, 2, 3, 4 \}, \{ 1, 2, 5, 6 \},\{ 3, 4, 5, 6 \} \}$, then zero of $4 \times 4$ minors of matrix $Pl_{\mathbb{T}}(D_{L_4})$ for same fixed $\mathbb{T}$ naturally give rise to relations among  relations in (\ref{rel:beta3}).  For example $(d)=0$ and $(e)=0$ correspond to vanishing of minors obtained considering, respectively, 1st, 3rd and 5th columns and 1st, 3rd and 6th columns of $A_\mathbb{T}(\mathcal{A}_\infty)$. Adding to $A_{\mathbb{T}}(\mathcal{A}_\infty)$ the normal vector to the hyperplane $D_{\{ 2,4,5,6 \}}$ as 4th row we get 
\begin{center}
$Pl_{\mathbb{T}}(D_{\{ 2,4,5,6 \}})=$
$\begin{pmatrix}
-\beta_{234} & \beta_{134} & -\beta_{124} & \beta_{123} & 0 & 0 \\
-\beta_{256} & \beta_{156} & 0 & 0 & -\beta_{126} & \beta_{125} \\
0 & 0 & -\beta_{456} & \beta_{356} & -\beta_{346} & \beta_{345} \\
0 & -\beta_{456} & 0 & \beta_{256} & -\beta_{246} & \beta_{245}
\end{pmatrix}$
\end{center}
and calculating the determinant of submatrix obtained by 1st, 3rd, 5th and 6th columns we get the relation among (e) and (d) : \begin{equation}\label{eq:syz1}
\beta_{246} \cdot (e) - \beta_{245} \cdot (d) =0 \quad .
\end{equation} 
Anagously vanishing of minor obtained by 1st, 4th, 5th and 6th columns gives: 
\begin{equation}\label{eq:syz2}
\beta_{256}\beta_{234} \cdot (c) - \beta_{246} \cdot (g) + \beta_{245} \cdot (f) =0 \quad .
\end{equation} 
Applying similar considerations to opportunely chosen $L_4 \notin \mathbb{T}$ we get the following additional syzygies. \\
Vanishing of minors obtained considering 1st, 4th, 5th and 6th columns and 1st, 3rd, 5th and 6th columns of 
\begin{equation*}
Pl_{\mathbb{T}}(D_{\{ 2,3,5,6 \}})=
\begin{pmatrix}
-\beta_{234} & \beta_{134} & -\beta_{124} & \beta_{123} & 0 & 0 \\
-\beta_{256} & \beta_{156} & 0 & 0 & -\beta_{126} & \beta_{125} \\
0 & 0 & -\beta_{456} & \beta_{356} & -\beta_{346} & \beta_{345} \\
0 & -\beta_{356} & \beta_{256} & 0 & -\beta_{236} & \beta_{235}
\end{pmatrix}
\end{equation*}
correspond, respectively, to relations $\beta_{236} \cdot (g) - \beta_{235} \cdot (f) =0$ and $\beta_{256}\beta_{234} \cdot (c) + \beta_{236} \cdot (e) - \beta_{235} \cdot (d) =0$. Those relations, jointly with the one in equations (\ref{eq:syz1}) and (\ref{eq:syz2}), state dependency of $(d), (e), (f)$ and $(g)$ from $(c)$ which, in turn, is equivalent to $(a)$, i.e. they are all zero if and only if $(a)$ is zero.\\
By vanishing of minors given by 2nd, 3rd, 5th and 6th columns and 2nd, 4th, 5th and 6th columns of submatrix
\begin{equation*}
Pl_{\mathbb{T}}(D_{\{ 1,4,5,6 \}})=
\begin{pmatrix}
-\beta_{234} & \beta_{134} & -\beta_{124} & \beta_{123} & 0 & 0 \\
-\beta_{256} & \beta_{156} & 0 & 0 & -\beta_{126} & \beta_{125} \\
0 & 0 & -\beta_{456} & \beta_{356} & -\beta_{346} & \beta_{345} \\
-\beta_{456} & 0 & 0 & \beta_{156} & -\beta_{146} & \beta_{145}
\end{pmatrix}
\end{equation*}
we get, respectively, $\beta_{146} \cdot (i) - \beta_{145} \cdot (h) =0$ and   $\beta_{156}\beta_{134} \cdot (c) - \beta_{146} \cdot (k) + \beta_{145} \cdot (j) =0$.  \\
Finally vanishing of minors given by 2nd, 4th, 5th and 6th columns and 2nd, 3rd, 5th and 6th columns of
\begin{equation*}
Pl_{\mathbb{T}}(D_{\{ 1,3,5,6 \}})=
\begin{pmatrix}
-\beta_{234} & \beta_{134} & -\beta_{124} & \beta_{123} & 0 & 0 \\
-\beta_{256} & \beta_{156} & 0 & 0 & -\beta_{126} & \beta_{125} \\
0 & 0 & -\beta_{456} & \beta_{356} & -\beta_{346} & \beta_{345} \\
-\beta_{356} & 0 & \beta_{156} & 0 & -\beta_{136} & \beta_{135}
\end{pmatrix}
\end{equation*}
give relations $\beta_{136} \cdot (k) - \beta_{135} \cdot (j) =0$ and $-\beta_{156}\beta_{134} \cdot (c) - \beta_{136} \cdot (i) + \beta_{135} \cdot (h) =0$. \\
That is relations in equation (\ref{rel:beta3}) are all equivalent and we are left with only one independent relation 
\begin{equation}
 (a)=0: \quad \beta_{134}\beta_{256} - \beta_{234}\beta_{156}=0.
\end{equation}
This degree 2 homogeneous polynomial defines a degree 2 hypersurface on the projective variety $Gr(3,6)$. \\
The above computations are a direct consequence of the following more general Lemma.

\begin{lem}\label{lem:van}
Let $A(\A_\infty)$ be the Pl\"{u}cker matrix associated to a generic arrangement $\A$ of $n$ hyperplanes in $\CC^3$ and $\TT$ a good 6-partition of indices $i_1, \ldots,  i_6 \in [n]$. If entries $\beta_I$ of the matrix $A(\A_\infty)$ satisfy Pl\"{u}cker relations, then $\rank A_\TT(\A_\infty) = 2$ if and only if one of its 3 $\times$ 3 minor vanishes.
\end{lem}

\begin{proof}
$\Rightarrow$) Since $\rank A_\TT(\A_\infty) = 2$ if and only if all 3 $\times$ 3 minors of $A_\TT(\A_\infty)$ vanish, it is obvious.

$\Leftarrow$) Entries $\beta_I$ of $A(\A_\infty)$ satisfy Pl\"{u}cker relations if and only if any $4 \times 4$ minor in $A(\A_\infty)$ vanishes. For any 4 columns $ s_1 < s_2 < s_3 < s_4  \in \{i_1, \ldots, i_6\}$ of matrix $A(\A_\infty)$ let $M_i$ and $M_j$ be the two  $3 \times 3$  minors in $A_\TT(\A_\infty)$  obtained considering, respectively, columns $\{ s_1, s_2 , s_3 , s_4 \} \backslash \{s_i\}$ and $\{ s_1 , s_2 , s_3 , s_4 \} \backslash \{s_j\}$. If we add to submatrix $A_\TT(\A_\infty)$ the row of $A(\A_\infty)$ corresponding to vector $\alpha_L$, $L = \{ s_i, s_j, s_5, s_6 \}$, with $\{ s_5, s_6 \} = \{i_1, \ldots, i_6\}\backslash \{ s_1, s_2, s_3, s_4 \}$, then the $4 \times 4$ minor of the matrix $\begin{pmatrix}
A_\TT(\A_\infty) \\
\alpha_L
\end{pmatrix}$
%\begin{equation}\label{eq:rank3}
%\begin{pmatrix}
%A_\TT(\A_\infty) \\
%\alpha_L
%\end{pmatrix}
%= 3 \quad ,
%\end{equation}
obtained considering columns $\{ s_1, s_2, s_3, s_4 \}$ vanishes, that is 
\begin{equation}
\beta_{L \backslash \{s_i\}} M_i \pm \beta_{L \backslash \{s_j\}} M_j = 0
\end{equation}
where $\beta_{L \backslash \{s_t\}}$ is the entry of the row $\alpha_L$ in the column $s_t, t=i,j$. Dividing by  $\beta_{L \backslash \{s_i\}} \neq 0$ (entries of $A(\A_\infty)$ are all not zero by $\A$ generic) we get 
\begin{equation}
M_i = \pm M_j \cdot \cfrac{\beta_{L \backslash \{s_j\}}}{\beta_{L \backslash \{s_i\}}}
\end{equation}
that is $M_i = 0$ if and only if $M_j =0$. Applying the above considerations to any subset $\{ s_1 < s_2 < s_3 < s_4 \} \subset \{i_1, \ldots, i_6\}$ and transitivity of equality, we get that if a 3 $\times$ 3 minor of $A(\A_\infty)$ vanishes then all minors vanish.
\end{proof}

\begin{rem}\label{rem:gen}
Recall that if $\A$ is an arrangement of $n$ hyperplanes in $\CC^3$ then the matrix $A(\mathcal{A}_\infty)$ is an ${n \choose 4} \times n$ matrix such that for any $L= \{s_1< s_2 < s_3 < s_4\}$, entries $(x_1, \ldots , x_n)$ of row vector $\alpha_L$ are all zeros except $x_{i_j}=(-1)^{j}\beta_{I_j}, I_j=L\setminus \{s_j\}, j=1,\ldots,4$. Hence for any fixed $6$ indices $s_1 < \ldots <s_6 \in [n]$ we get a ${6 \choose 4}\times 6$ submatrix of $A(\mathcal{A}_\infty)$ obtained considering all rows $\alpha_L$, $L \subset \{s_1, \ldots, s_6\}, \mid L\mid =4$ and columns $\{s_1, \ldots, s_6\}$ ( all columns $j \notin \{s_1, \ldots, s_6\}$ of the matrix $(\alpha_L)_{L \subset \{s_1, \ldots, s_6\}, \mid L\mid =4}$ are zero). It follows that the general case of $n$ hyperplanes in $\CC^3$ essentially reduce to the case $n=6$. \\
On the other hand it is an easy remark that, if $s_1 < \ldots <s_6 \in [n]$ are $6$ fixed indices and $\mathbb{T} = \{ \{s_1, s_2, s_3, s_4\}$, $\{s_1, s_2, s_5, s_6\}, \{s_3, s_4, s_5, s_6\} \}$ ( analogous of good $6$-partition $\{ \{1, 2, 3, 4\}$, $\{1, 2, 5, 6\}, \{3, 4, 5, 6\} \}$ of indices $\{1, \ldots 6\}$ ), then any other good $6$-partition on indices $\{s_1 , \ldots , s_6\}$ is of the form
\begin{equation}\label{eq:notT}
\sigma.\mathbb{T}=\{ \{i_1, i_2, i_3, i_4\}, \{i_1, i_2, i_5, i_6\}, \{i_3, i_4, i_5, i_6\} \}
\end{equation}
where $i_j=\sigma(s_j)$, $\sigma \in \mathbf{S_6}$, $\mathbf{S_6}$ being the group of all permutations of indices $\{s_1, \ldots , s_6\}$. Notice that in general $i_j$ are not ordered and we can have $i_j > i_{j+1}$.
\end{rem}

The following Lemma holds.

\begin{lem}\label{lem:fin} Let $\A$ be an arrangement of $n$ hyperplanes in $\CC^3$ and $\sigma.\mathbb{T} = \{ \{i_1, i_2, i_3, i_4\}$, $\{i_1, i_2, i_5, i_6\}, \{i_3, i_4, i_5, i_6\} \}$, a good $6$-partition of indices $s_1 < \ldots <s_6 \in [n]$ such that $\rank A_{\sigma.\mathbb{T}}(\mathcal{A}_\infty)=2$ then $\A$ is a point in the hypersurface 
\begin{equation}
 \beta_{i_1 i_3 i_4} \beta_{i_2 i_5 i_6} - \beta_{i_2 i_3 i_4} \beta_{i_1 i_5 i_6} =0 \quad .
\end{equation}
\end{lem}

\begin{proof} Let $\sigma.\mathbb{T} = \{ L_1'=\{i_1, i_2, i_3, i_4\}, L_2'=\{i_1, i_2, i_5, i_6\}, L_3'=\{i_3, i_4, i_5, i_6\} \}$ be a good $6$-partition of indices $s_1 < \ldots <s_6 \in [n]$ and denote by $(L_1')=(i_1, i_2, i_3, i_4)$, $(L_2')=(i_1, i_2, i_5, i_6)$ and $(L_3')=(i_3, i_4, i_5, i_6)$ the ordered $4$-uples of indices. Then there exist unique permutations $\tau_i$, $i=1,2,3$ of indices $s_1 < \ldots <s_6$ such that $\tau_i$ fixes indices outside $L_i'$ and, if $L_i'=\{s_{j_1}< s_{j_2}< s_{j_3}< s_{j_4}\} $, then $(L_i')=\tau_i.L_i'=(\tau_i(s_{j_1}),\tau_i(s_{j_2}),\tau_i(s_{j_3}),\tau_i(s_{j_4}))$, $i=1,2,3$. By determinant rule on permutations of columns we have that 
\begin{eqnarray*}
\sum_{j=1}^4 (-1)^j \det (\alpha_{\tau(1)}, \dots, \hat{\alpha_{\tau(j)}}, \dots, \alpha_{\tau(4)})e_{\tau(j)} &=& 
    \begin{vmatrix}
      a_{\tau(1)1} & a_{\tau(2)1} & a_{\tau(3)1} & a_{\tau(4)1} \\
      a_{\tau(1)2} & a_{\tau(2)2} & a_{\tau(3)2} & a_{\tau(4)2} \\
      a_{\tau(1)3} & a_{\tau(2)3} & a_{\tau(3)3} & a_{\tau(4)3} \\
       e_{\tau(1)} &  e_{\tau(2)}  &  e_{\tau(3)}  &   e_{\tau(4)}
    \end{vmatrix} \\
&=& \sign(\tau)
    \begin{vmatrix}
      a_{11} & a_{21} & a_{31} & a_{41} \\
      a_{12} & a_{22} & a_{32} & a_{42} \\
      a_{13} & a_{23} & a_{33} & a_{43} \\
      e_{1} &  e_{2}  &  e_{3}  &   e_{4}
    \end{vmatrix} \\
&=& \sign(\tau)\sum_{j=1}^4 (-1)^j \det (\alpha_{1}, \dots, \hat{\alpha_{j}}, \dots, \alpha_{4})e_{j} \quad.
\end{eqnarray*}
Hence, if we define the matrix $\sigma.A_\TT$ as the matrix having in its rows respectively the coefficients of the three vectors 
\begin{eqnarray*}
\tau_1.\alpha_{L_1'} &=& \sum_{j=1}^4 (-1)^j \det (\alpha_{i_1}, \dots, \hat{\alpha_{i_j}}, \dots, \alpha_{i_4})e_{i_j}, \\
\tau_2.\alpha_{L_2'} &=& \sum_{j \in \{1,2,5,6\}} (-1)^j \det (\alpha_{i_1}, \dots, \hat{ \alpha_{i_j}}, \dots, \alpha_{i_6})e_{i_j}, \\
\tau_3.\alpha_{L_3'} &=& \sum_{j=3}^6 (-1)^j \det (\alpha_{i_3}, \dots, \hat{\alpha_{i_j}}, \dots, \alpha_{i_6})e_{i_j}
\end{eqnarray*}
with respect to the ordered basis $\{e_{i_1}, \ldots, e_{i_6}\}$, then $i$-th row of $\sigma .A_\TT$ is obtained from $i$-th row of $A_{\sigma.\mathbb{T}}(\mathcal{A}_\infty)$ by permutation $\sigma$ of columns and multiplication by $\sign(\tau_i)$ (notice that $\sigma_{\mid_{L_i'}}=\tau_i$ ). That is $\rank \sigma .A_\TT= \rank A_{\sigma.\mathbb{T}}(\mathcal{A}_\infty)$ and, more in details, the $3 \times 3$ minor given by columns $\{i,j,k\}$ in $A_{\sigma.\mathbb{T}}(\mathcal{A}_\infty)$ vanishes if and only if the  $3 \times 3$ minor of columns $\{\sigma(i),\sigma(j),\sigma(k)\}$ in $\sigma.A_\TT$ vanishes. Hence, by Lemma \ref{lem:van}  $\rank A_{\sigma.\mathbb{T}}(\mathcal{A}_\infty)=\rank \sigma .A_\TT=2$  if and only if one minor vanishes. In particular the first three columns $\{i_1,i_2,i_3\}$ in $\sigma .A_\TT$ are of the form
\begin{equation*}
\begin{pmatrix}
-\beta_{i_2 i_3 i_4} \\
-\beta_{i_2 i_5 i_6}  \\
0 
\end{pmatrix} \qquad
\begin{pmatrix}
\beta_{i_1 i_3 i_ 4}\\
\beta_{i_1 i_5 i_6}\\
0
\end{pmatrix} \qquad
\begin{pmatrix}
-\beta_{i_1 i_2 i_4}\\
0\\
-\beta_{i_4 i_5 i_6}
\end{pmatrix}
\end{equation*}
from which we get that the $3 \times 3$ minor corresponding to them vanishes if and only if $$\beta_{i_1 i_3 i_4} \beta_{i_2 i_5 i_6} - \beta_{i_2 i_3 i_4} \beta_{i_1 i_5 i_6} =0$$ 
( recall that all entries $\beta_I $ in the matrix $A(\mathcal{A}_\infty)$\ verify $\beta_I \neq 0$ ). 
\end{proof}

By Remark \ref{rem:gen} and Lemma \ref{lem:fin}, the following main Theorem follows.

\begin{thm}\label{thm:gr}
The set of generic arrangements $\A$ of $n$ hyperplanes in $\CC^3$ that contains a dependent sub-arrangement is the set of points in an hypersurface in Grassmannian $Gr(3, n)$ such that each component is intersection of Grassmannian with a quadric.
\end{thm}

%\begin{thm}\label{thm:gr}A generic arrangement $\A$ of $n$ hyperplanes in $\CC^3$ that contains a dependent subarrangement is a point in a degree 2 hypersurface in the Grassmannian $Gr(3, n)$.
%\end{thm}

\end{document}